\documentclass[10pt]{article}
\usepackage{pgfplots}
\pgfplotsset{width=7cm,compat=1.8}
\usetikzlibrary{patterns}
\usepackage[latin1]{inputenc}
\usepackage{latexsym}
\usepackage{amsmath,amssymb,amsfonts,amsthm}
\usepackage{xcolor}
\usepackage{mathrsfs}
\usepackage{comment}
\usepackage{bbm}
\usepackage{paralist}
\usepackage{mathtools}

\usepackage{tikz}

\usetikzlibrary{calc}
\usetikzlibrary{decorations.pathreplacing}
\usetikzlibrary{arrows}
\usepgfplotslibrary{fillbetween}

\usepackage[toc,page]{appendix}

\usepackage{enumerate}

\usepackage[numbers,comma,sort]{natbib}

\definecolor{db}{RGB}{0, 0, 130}
\usepackage[colorlinks=true,citecolor=db,linkcolor=black,urlcolor=blue,pdfstartview=FitH]{hyperref}

\usepackage{pdfsync}

\definecolor{rp}{rgb}{0.25, 0, 0.75}
\definecolor{dg}{rgb}{0, 0.5, 0}

\newcommand{\bqn}{\begin{equation}}
\newcommand{\eqn}{\end{equation}}
\newcommand{\bqne}{\begin{equation*}}
\newcommand{\eqne}{\end{equation*}}

\numberwithin{equation}{section}

\usepackage{todonotes}

\makeatletter
\newcommand{\customlabel}[2]{%
   \protected@write \@auxout {}{\string\newlabel {#1}{{#2}{\thepage}{#2}{#1}{}}}%
   \hypertarget{#1}{#2\hspace{-0.14cm}}
}
\makeatother

\newtheorem{definition}{Definition}[section]

\newtheorem{lemma}[definition]{Lemma}

\newtheorem{proposition}[definition]{Proposition}

\newtheorem{remark}[definition]{Remark}

\author{Lukas Anzeletti\footnote{Universit\'e Paris-Saclay, CentraleSup\'elec, MICS and CNRS FR-3487;   \texttt{lukas.anzeletti@centralesupelec.fr} Acknowledging the support of the Labex de Math\'ematique Hadamard.}}

\title{ \Large{\textbf{Comparison of classical and path-by-path solutions to SDEs}}}
\begin{document}

\maketitle

\begin{abstract}
We consider the Stochastic Differential Equation \(X_t = X_0 + \int_0^t
	b(s,X_s) ds + B_t,\) in \(\mathbb{R}^d\). We give an example of a drift \(b\) such that 
	there does not exist a weak solution, but there exists a solution for almost every 
	realization of the Brownian motion \(B\). We also give an explicit example of a drift such that the 
	SDE  has a pathwise unique weak solution, but path-by-path uniqueness (i.e. uniqueness of solutions to the ODE for almost every realization of the Brownian motion) is lost.
	These counterexamples extend the results obtained in \cite{ShaWre} to dimension 
	$d=1$.
\end{abstract}

\noindent\textit{\textbf{Keywords and phrases:} Regularization by noise, Singular drifts, Path-by-path uniqueness.} 

\medskip

\noindent\textbf{MSC2020 subject classification: 60H50, 60H10, 60J65}  .

\setcounter{tocdepth}{2}
\renewcommand\contentsname{}

\section{Introduction}

Let $b\colon [0,T]\times \mathbb{R}^d \rightarrow \mathbb{R}^d$ be measurable and $B$ be 
a $d$-dimensional Brownian motion. Throughout the paper we always consider Stochastic 
Differential Equations (SDEs) of the type
\begin{equation} \label{eq:SDE}
		X_t=X_0+\int_0^t b(s,X_s) ds + B_t.
\end{equation}
 
	One can think of different kinds of 
	solutions:
\begin{enumerate}[(1)]
	\item \label{en:firsttype} Weak and strong solutions to the SDE 
	\eqref{eq:besselbridge}. 
	\item \label{en:secondtype} Considering the equation \eqref{eq:SDE} as an ODE for each realization \((B_t(\omega))_{0\leqslant t \leqslant T}\) of the 
	Brownian path (``path-by-path" solutions). 
\end{enumerate}

The second of the above is only possible as we consider an SDE with constant diffusion 
coefficient. Therefore we can evaluate the noise term therein for each Brownian path, which 
would not be possible with a more general diffusion term as this would lead to a stochastic 
integral. The aim of this paper is to show via two counterexamples that, for any dimension $d\geqslant 1$, the two notions above 
are not equivalent.

Recall that there are drifts $b$ such that \eqref{eq:SDE} is well-posed, although the equation without additive noise is not. This phenomenon is called regularization by noise (see \cite{Flandoli} for a thorough presentation, in particular on PDE models of fluid dynamics). 
Considering \eqref{eq:SDE} in the sense of (\ref{en:firsttype}), there is an extensive literature which we will not describe thoroughly. Nevertheless let us mention the important works on existence and uniqueness of solutions by \citet{Veretennikov} for bounded measurable drift and by \citet{KrylovRockner} for drifts fulfilling an $L^p-L^q$ condition. For further works on distributional drifts see \cite{FlandoliRussoWolf, BassChen, DelarueDiel, FIR, LeGall}.

The following was shown by \citet{Davie} when considering \eqref{eq:SDE} in the sense of (\ref{en:secondtype}): For bounded measurable drift $b$, there exists a unique solution to \eqref{eq:SDE} for almost every realization of Brownian motion, considered as a purely deterministic integral equation (i.e. ``path-by-path" uniqueness). After the initial result on path-by-path uniqueness in \cite{Davie}, several extensions have been proven. In \cite{Shapo, shapocorrection} the author proved Davie's Theorem using different techniques (in particular the flow property of strong solutions) and also recovered uniqueness for a class of unbounded measurable drifts. For yet another approach, linking \eqref{eq:SDE} to a backward SPDE, see \cite{BeckFlandoli}. For results on ``path-by-path" uniqueness for SPDEs see \cite{ButMyt, BeckFlandoli}. There are also results on existence and uniqueness of path-by-path solutions available for SDEs with fractional Brownian noise, allowing for distributional drifts (see \cite{AnRiTa,CatellierGubinelli,GaleatiGubinelli}). 

In both \cite{Flandoli} and \cite{Alabert} it was phrased as an open problem whether every solution of type \eqref{en:secondtype} has to arise from an adapted solution. In a higher-dimensional setting this question was answered via a counterexample in \cite{ShaWre}, making heavy use of the dimension $d\geqslant 2$. Hence, it is a natural question whether similar counterexamples can be constructed in a one-dimensional setting. We give a positive answer to this question in Section~\ref{counterexample}. The construction in Section~\ref{counterexample} can easily be extended to the case $d\geqslant 2$.

\paragraph{Notations.}
\begin{compactitem}

\item We use the notation \(\mathbb{F} = (\mathcal{F}_t)_{t \geqslant 
	0}\)
for a filtration. All filtrations are assumed to fulfill the usual conditions.

\item For a $\sigma$-algebra $\mathcal{F}$ and a stopping time $\tau$ we define the stopping time $\sigma$-algebra by $\mathcal{F}_{\tau}\coloneqq \{A \in \mathcal{F}: A \cap \{\tau\leqslant t\} \in \mathcal{F}_t, \, \forall t\geqslant 0 \}$.

\item We call $(B_t)_{t \geqslant 
	0}$ an $\mathbb{F}$-Brownian 
motion if $B$ is a Brownian motion adapted to $\mathbb{F}$ and for any $0 
\leqslant s \leqslant t$, $B_t-B_s$ is independent of $\mathcal{F}_s$.

\end{compactitem} 

\paragraph{Definitions and previous results.}

\begin{definition}[Existence]
\begin{enumerate}[(i)]
\item If there exists a filtered probability 
space $(\Omega,\mathcal{F},\mathbb{F},\mathbb{P})$ equipped with a Brownian motion \(B\) 
and an \(\mathbb{F}\)-adapted process \(X\) such that $(X,B)$ fulfills \eqref{eq:SDE} almost 
surely, we say that \((X,B)\) is a weak solution to \eqref{eq:SDE}. If the choice of $B$ is clear from the context, we write that $X$ is a weak solution.

\item We call $X$ a strong solution if $X$ is a weak solution and $X$ is adapted w.r.t. the filtration generated by $B$.

\item Let $(\Omega,\mathcal{F},\mathbb{F},\mathbb{P})$ be a filtered probability space on which a Brownian motion $B$ is defined. We call a mapping $X\colon\Omega \rightarrow \mathcal{C}([0,T])$ a path-by-path solution if there exists a set $\tilde{\Omega}\subset \Omega$ of full measure such that, for all $\omega \in \tilde{\Omega}$, $(X(\omega),B(\omega))$ fulfills equation \eqref{eq:SDE}.
\end{enumerate}

\end{definition}
\begin{definition}[Uniqueness]
\begin{enumerate}[(i)]
\item We say that pathwise uniqueness for \eqref{eq:SDE} holds if for any two weak solutions $(X,B), (\tilde{X},B)$ defined on the same filtered probability space with the same Brownian motion $B$ and the same initial condition, $X$ and $\tilde{X}$ are indistinguishable.

\item We say that path-by-path uniqueness for \eqref{eq:SDE} holds if for any probability space on which a Brownian motion $B$ is defined, there exists a null-set $\mathcal{N}$ such that for $\omega \notin \mathcal{N}$, there exists a unique solution to 
\begin{equation*}
X_t(\omega)= X_0(\omega)+\int_0^t b(s,X_s(\omega))ds + B_t(\omega).
\end{equation*} 

\item Let $X$  be a nonnegative weak solution to \eqref{eq:SDE} such that all other nonnegative weak solutions defined on the same probability space with the same Brownian motion and the same initial condition are indistinguishable from $X$. Then we call $X$ the pathwise unique nonnegative weak solution. In a symmetric manner we define a pathwise unique nonpositive weak solution.

\end{enumerate}
\end{definition}

The following implications follow directly from the definitions:
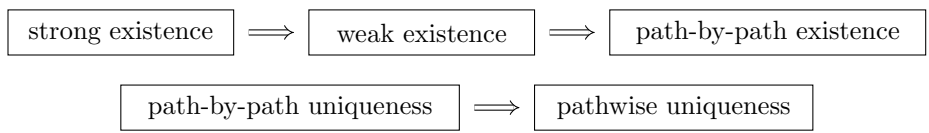
\begin{center}
\begin{tikzpicture}
\draw (0,-0.3) rectangle (3,0.3) node[pos=.5] {strong existence};
\draw (4,-0.3) rectangle (7,0.3) node[pos=.5] {weak existence};
\draw (8,-0.3) rectangle (12.2,0.3) node[pos=.5] {path-by-path existence};
\draw[-implies,double equal sign distance] (3.2,0) -- (3.8,0);
\draw[-implies,double equal sign distance] (7.2,0) -- (7.8,0);
\draw (1.5,-1.3) rectangle (6,-0.7) node[pos=.5] {path-by-path uniqueness};
\draw[-implies,double equal sign distance] (6.2,-1) -- (6.8,-1);
\draw (7,-1.3) rectangle (10.7,-0.7) node[pos=.5] {pathwise uniqueness};
\end{tikzpicture}
\end{center}

Having the proper definitions of different kinds of solutions at hand, we can precisely formulate the initial result on path-by-path uniqueness to \eqref{eq:SDE} and the counterexamples mentioned earlier. 

For a bounded measurable drift the following is known: Let $B$ be a $d$-dimensional Brownian motion defined on a filtered probability space $(\Omega,\mathcal{F},\mathbb{F},\mathbb{P})$, $X_0 \in \mathbb{R}$ and let $b\colon[0,\infty) \times \mathbb{R}^d \rightarrow \mathbb{R}^d$ be a bounded Borel-measurable function. In \cite{Davie} it was shown that
\begin{align} \label{Davie}
    X_t=X_0 + \int_0^t b(r,X_r) dr+ B_t
\end{align}
has a unique path-by-path solution.

In \cite{ShaWre}, for dimension $d\geqslant 2$, drifts $b$ are constructed such that
\begin{itemize}
\item there is existence of path-by-path solutions, but there exists no weak solution;
\item there exists a pathwise unique weak solution to \eqref{eq:SDE}, but path-by-path uniqueness is lost.
\end{itemize}

In Section~\ref{counterexample}, we will construct drifts $b$ such that the two situations described above happen for Equation \eqref{eq:SDE}. Recall that by \citet{KrylovRockner} strong existence and pathwise uniqueness for Equation \eqref{eq:SDE} hold for $b \in L^q([0,T],L^p(\mathbb{R}^d))$ where $p\in [2,\infty],q \in (2,\infty]$ with
\begin{equation} \label{eq:pq}
\frac{d}{p}+\frac{2}{q}<1.
\end{equation}
Hence, for the first of the above counterexamples, $b$ cannot be in this class of functions. Although a priori this (with the current results on path-by-path uniqueness) does not have to be the case for the second counterexample above, we are not aware of a counterexample such that $b\in L^q([0,T],L^p(\mathbb{R}^d))$ for $p\in [2,\infty],q \in (2,\infty]$ fulfilling \eqref{eq:pq}.

\section{Preparatory results and Bessel bridges}

\begin{lemma} \label{lem:noweaksolution}

If \eqref{eq:SDE} has a weak solution with initial condition $X_0$ fulfilling $\mathbb{P}(X_0=0)>0$, then it also has a weak solution with $X_0=0$.
\end{lemma}

\begin{proof}
Let $(X,B)$, defined on a filtered probability space $(\Omega,\mathcal{F},\mathbb{F},\mathbb{P})$, be a weak solution to \eqref{eq:SDE} with $X_0$ fulfilling $\mathbb{P}(X_0=0)>0$. Let $A\coloneqq \{X_0=0\}$. Define $(\tilde{X},\tilde{B})$ and $(\tilde{\Omega},\tilde{\mathcal{F}},\tilde{\mathbb{F}},\tilde{\mathbb{P}})$ in the following way: $\tilde{\Omega}=A$; $\tilde{\mathcal{F}}=\{E\cap A:E \in \mathcal{F}\}$; $\tilde{\mathcal{F}}_t=\{E\cap A:E \in \mathcal{F}_t\}$; for $C\in \tilde{\mathcal{F}}$, $\tilde{\mathbb{P}}(C)=\frac{\mathbb{P}(C)}{\mathbb{P}(A)}$; for $\omega \in A$, $\tilde{X}(\omega)=X(\omega)$ and $\tilde{B}(\omega)=B(\omega)$. First, we check that $\tilde{X}$ is adapted with respect to $\tilde{\mathbb{F}}$. Let $U \subset \mathbb{R}$ be measurable. Then 
\begin{equation*}
    \tilde{X}_t^{-1}(U)=X_t^{-1}(U) \cap A \in \tilde{\mathcal{F}}_t.
\end{equation*}
Next, we check that $\tilde{B}$ is an $\tilde{\mathbb{F}}$-Brownian motion. As $X_0$ is $\mathcal{F}_0$-measurable and $B$ is an $\mathbb{F}$-Brownian motion, we know that, for $0\leqslant s \leqslant t$, 
\begin{equation*}
    \tilde{\mathbb{P}}(\tilde{B}_t-\tilde{B}_s \in U)=\frac{\mathbb{P}(\{B_t-B_s \in U\}\cap A)}{\mathbb{P}(A)}=\mathbb{P}(B_t-B_s \in U)
\end{equation*}
and using this we get, for $V \in \tilde{\mathcal{F}}_s$,
\begin{equation*}
    \tilde{\mathbb{P}}\left(\{\tilde{B}_t-\tilde{B}_s\in U\}\cap V \right)=\frac{\mathbb{P}(\{B_t-B_s \in U\}\cap V)}{\mathbb{P}(A)}=\tilde{\mathbb{P}}(\tilde{B}_t-\tilde{B}_s\in U)\tilde{\mathbb{P}}(V).
\end{equation*}
Hence, $\tilde{B}$ is an $\tilde{\mathbb{F}}$-Brownian motion. 

Clearly  $(\tilde{X},\tilde{B})$ fulfills equation \eqref{eq:SDE} a.s. 
\end{proof}

The following lemma and its proof are similar to \cite[Example 7.4]{BeckFlandoli} and \cite[Example 2.1]{Cherny2002}.

\begin{lemma} \label{lem:noweaksolutiondirect2}
Let $f\colon\mathbb{R}\rightarrow \mathbb{R}$ be measurable. There does not exist a weak solution to
\begin{equation}
    dX_t= b(X_t) dt+dB_t,\quad t \in [0,1], \label{eq:nosolutiondirect2}
\end{equation}
with initial condition $X_0$ fulfilling $\mathbb{P}(X_0 \in (-1,1))>0$ and 
\begin{equation*}
b(x)\coloneqq -\mathbbm{1}_{\{x \neq 0, |x|\leqslant 1\}}\frac{1}{2x} + \mathbbm{1}_{\{|x|> 1\}} f(x).
\end{equation*}
\end{lemma}

\begin{proof}
\textbf{Step 1:}
First we show that there exists no weak solution to \eqref{eq:nosolutiondirect2} on $[0,1/2]$ with initial condition $X_0=0$. Assume there exists a weak solution $(X,B)$. Let ${\tau_1}\coloneqq \inf\{s\geqslant 0: |X_s|\geqslant 1\}$. Then by It\^{o}'s Lemma, we get that
\begin{align*}
	X_{t\wedge {\tau_1}}^2&=\int_0^{t\wedge {\tau_1}} - \mathbbm{1}_{\{X_s \neq 0\}} ds + \int_0^{t \wedge {\tau_1}} ds + \int_0^{t \wedge {\tau_1}} 2 X_s dB_s \\
	&=\int_0^{t\wedge {\tau_1}} \mathbbm{1}_{\{X_s=0\}} ds + \int_0^{t\wedge {\tau_1}} 2 X_s dB_s.
\end{align*}
Define $X^{\tau_1}_t\coloneqq X_{t\wedge {\tau_1}}$. Note that $X^{\tau_1}$ is a continuous semimartingale with quadratic variation $\langle X^{\tau_1} \rangle_t=t\wedge {\tau_1}$, so we can use \cite[Chapter 6, Corollary (1.6)]{RevuzYor} to obtain
\begin{equation*}
    \int_0^{t\wedge {\tau_1}} \mathbbm{1}_{\{X_s=0\}} ds = \int_0^{t\wedge {\tau_1}} \mathbbm{1}_{\{X_s=0\}} d \langle X \rangle_s = \int_\mathbb{R} \mathbbm{1}_{\{x=0\}} L_{t\wedge {\tau_1}}^x(X) dx = 0,
\end{equation*}
where $L(X)$ denotes the local time of $X$.
Hence, $(X^{\tau_1}_t)^2=\int_0^{t\wedge {\tau_1}} 2X_s dB_s$ is a local martingale and as $|X^{\tau_1}|\leqslant 1$ it is also a martingale. Note that $X^{\tau_1}_0=0$ and $(X^{\tau_1})^2\geqslant 0$. This implies that $X^{\tau_1}$ must be identically $0$, which contradicts $\langle X^{\tau_1} \rangle_t=t\wedge {\tau_1}$ as ${\tau_1}>0$ a.s.

\textbf{Step 2:}
Assume that there exists a weak solution $(X,B)$ to \eqref{eq:nosolutiondirect2} (defined on some filtered probability space $(\Omega,\mathcal{F},\mathbb{F},\mathbb{P})$) with initial condition $X_0$ fulfilling $\mathbb{P}(X_0 \in (-1,1))>0$. We assume w.l.o.g. that $\mathbb{P}(X_0 \in [0,1))>0$ as $\mathbb{P}(X_0 \in (-1,0))>0$ can be dealt with in an entirely symmetric way. Let $a \in [0,1)$ such that $\mathbb{P}(X_0 \in [0,a])>0$. Let ${\tau_0} \coloneqq \inf \{t\geqslant 0:X_t=0\}$ and 
\begin{equation*}
C \coloneqq  \{X_0 \in [0,a]\} \cap \{\sup_{t \in [0,1/2]} B_{t}<1-a, B_{1/2}\leqslant -a\}.
\end{equation*}
Then $\{{\tau_0}\leqslant 1/2\} \supset C$. To see this let $\omega \in C$ and assume that $\inf_{t \in [0,1/2]} X_t(\omega)>0$. Then
\begin{align*}
X_{1/2}(\omega)=X_0(\omega)-\int_0^{1/2} \frac{1}{2X_s(\omega)} ds + B_{1/2}(\omega)\leqslant a+B_{1/2}(\omega)\leqslant 0,
\end{align*}
which gives a contradiction. Hence $\mathbb{P}({\tau_0}\leqslant 1/2)>0$ by continuity of $X$ and as $\mathbb{P}(C)>0$. Knowing this, below we construct a solution to \eqref{eq:nosolutiondirect2} on $[0,1/2]$ starting from $0$ with positive probability. This gives a contradiction to Step 1 due to Lemma~\ref{lem:noweaksolution}.

Let ${\tilde{\tau}_0}\coloneqq \tau_0 \wedge 1/2$, $\tilde{B}_{t}\coloneqq B_{{\tilde{\tau}_0}+t}-B_{\tilde{\tau}_0}$ and $\tilde{X}_{t}\coloneqq X_{{\tilde{\tau}_0}+t}$. Then $(\tilde{X},\tilde{B})$ is a weak solution to \eqref{eq:nosolutiondirect2} on $[0,1/2]$ with initial condition $X_{{\tilde{\tau}_0}}$ which fulfills $\mathbb{P}(X_{{\tilde{\tau}_0}}=0)>0$. First, for $t \in [0,1/2]$,
\begin{align*}
\tilde{X}_{t}=X_{{\tilde{\tau}_0}+t}&=X_{{\tilde{\tau}_0}}+\int_{{\tilde{\tau}_0}}^{{\tilde{\tau}_0}+t}b(X_s) ds + B_{{\tilde{\tau}_0}+t}-B_{{\tilde{\tau}_0}}\\
&=X_{{\tilde{\tau}_0}}+\int_0^t b(\tilde{X}_s) ds +\tilde{B}_t.
\end{align*}

Furthermore, $\tilde{B}$ is an $\mathbb{F}^{{\tilde{\tau}_0}}$-Brownian motion for $\mathcal{F}^{{\tilde{\tau}_0}}_t\coloneqq \mathcal{F}_{{\tilde{\tau}_0}+t}$ by \cite[page 23, Theorem 32]{Protter} and $\tilde{X}$ is clearly adapted to $\mathbb{F}^{\tilde{\tau}_0}$.
\end{proof}

In the following lemma we consider two SDEs and present some properties of solutions to these SDEs. Both have been investigated thoroughly in the literature. The lemma is stated to give a clearer presentation of the counterexamples in Section \ref{counterexample}.

\begin{lemma} \label{lem:allequations}
Let $y>0$. Consider the following equations on $[0,1]$: 
\begin{align}
    dX_t&= \mathbbm{1}_{\{X_t>0\}}\left(\frac{y-X_t}{1-t}+\frac{1}{X_t}\right) dt + \mathbbm{1}_{\{X_t<0\}}\left(\frac{-y-X_t}{1-t}+\frac{1}{X_t}\right)dt+dB_t, \quad X_0=0, \label{eq:besselbridge}\\
    dX_t&= \mathbbm{1}_{\{X_t \neq 0\}} \frac{1}{X_t} dt + dB_t, \quad X_0 \in \mathbb{R}. \label{eq:3bessel}
\end{align}

\begin{enumerate}[(a)]

\item Any weak solution to \eqref{eq:besselbridge} satisfies $|X_1|=y$ a.s. and does not change its sign on the interval $(0,1]$ a.s. Moreover, there exists a pathwise unique nonnegative strong solution and a pathwise unique nonpositive strong solution. We will call these solutions nonnegative/nonpositive Bessel bridge. \label{en:b}
\item For $X_0 \geqslant 0$, there exists a pathwise unique nonnegative strong solution to \eqref{eq:3bessel}. For $X_0>0$, pathwise uniqueness holds. \label{en:unique}
\end{enumerate}
\end{lemma}

\begin{proof}
\eqref{en:b}: Weak existence follows by \cite[page 274, Equation (29)]{Pitman}. Pathwise uniqueness and strong existence follow by the same arguments as in the proof of Proposition 1 in \cite{ShaWre}, which is stated for $y=1$.

\eqref{en:unique}: Special case of Theorem 3.2 in \cite{Cherny2000}.
\end{proof}

\begin{lemma} 
Let $X$ be the nonnegative Bessel bridge on $[0,1]$ with $X_0=0$ and $X_1=1$. Then there exists $\varepsilon>0$ small enough such that $\{\sup_{t \in [0,1]} X_t<2\}\supset \{\sup_{s \in [0,1]} |B_s|<\varepsilon\}$ and therefore
\begin{align*}
\mathbb{P}(\sup_{t \in [0,1]} X_t<2)>0.
\end{align*}
\end{lemma}

\begin{proof}
Recall that $X$ satisfies
\begin{align*}
X_t= \int_0^t \mathbbm{1}_{\{X_s>0\}} \left(\frac{1-X_s}{1-s}+\frac{1}{X_s}\right) ds+B_t.
\end{align*}
Let $\varepsilon\in (0,1/6)$ and consider the set $A\coloneqq \{\sup_{s \in [0,1]} |B_s|<\varepsilon\}$ fulfilling $\mathbb{P}(A)>0$. Let $\omega \in A$. Assume that there exists $\tau_2 \in [0,1]$ such that $X_{\tau_2}(\omega)=2$. Let $\tau_{5/3}\coloneqq \sup \{s \in [0,\tau_2]:X_s=5/3\}$. By continuity $X_{\tau_{5/3}}=5/3$. Then
\begin{align*}
X_{\tau_2}&= \frac{5}{3}+ \int_{\tau_{5/3}}^{\tau_2} \left(\frac{1-X_s}{1-s}+\frac{1}{X_s}\right) ds + B_{\tau_2}-B_{\tau_{5/3}}\\
&\leqslant \frac{5}{3}+\int_{\tau_{5/3}}^{\tau_2} \left(\frac{-2}{3}+\frac{3}{5} \right) ds+2\varepsilon<2,
\end{align*}
which gives a contradiction. Hence, $\sup_{t \in [0,1]} X_t(\omega)<2$ for $\omega \in A$.
\end{proof}

\begin{remark}\label{rem:upperbound}
Let $X$ be the Bessel bridge with $X_0=2$ and $X_1=1$. By symmetry and a space shift the above lemma gives that $\{\inf_{t \in [0,1]} X_t>0\}\supset \{\sup_{t \in [0,1]} |B_t|<\varepsilon\}$ for small enough $\varepsilon>0$ and therefore 
\begin{equation*}
\mathbb{P}(\inf_{t \in [0,1]} X_t>0)>0.
\end{equation*}
\end{remark}

\section{Counterexample} \label{counterexample}

\subsection{Path-by-path existence but no weak existence}

In the following proposition we construct an SDE without weak solutions, but with multiple path-by-path solutions. The construction is done in a one-dimensional setting, but can easily be extended to multiple dimensions.

\begin{proposition} \label{lem:counterexample1}
Consider the SDE
\begin{equation}\label{eq:eqofinterest}
    dX_t= b(t,X_t) dt + dB_t, \quad X_0=0, \quad t \in [0,3],
\end{equation}
where
\[
    b(t,x)=\begin{dcases}
        \mathbbm{1}_{\{x>0\}}\left(\frac{1-x}{1-t}+\frac{1}{x}\right) + \mathbbm{1}_{\{x<0\}}\left(\frac{-1-x}{1-t}+\frac{1}{x}\right) & \text{ if }0\leqslant t < 1, \\
        0 & \text{ if } 1 \leqslant t < 2, \\
        \mathbbm{1}_{\{x \neq 0, |x|\leqslant 1\}}\frac{-1}{2x} + \mathbbm{1}_{\{x> 1\}} \frac{1}{x-1} + \mathbbm{1}_{\{x < -1\}} \frac{1}{x+1} & \text{ if } 2\leqslant t \leqslant 3. \\
    \end{dcases}
\]

There does not exist a weak solution to \eqref{eq:eqofinterest}, but there exist path-by-path solutions.

\end{proposition}

\textbf{Idea of the proof.}

The drift is constructed in a way such that for a solution $X$ the following must hold:
$X_1$ is forced to be equal to $1$ or $-1$. On the time interval $[1,2]$ we let a Brownian motion evolve freely without any drift. On the time interval $[2,3]$ the drift is constructed in a way such that there exists no adapted solution if $X_2 \in (-1,1)$. However if $|X_2|>1$, then $X$ can be extended to the interval $[0,3]$ with $|X_t|>1$ for all $t \in [2,3]$ while still being adapted. Hence, any adapted solution must avoid taking a value in $(-1,1)$ at time $2$. If $B_2-B_1=X_2-X_1\in (0,2)$, this can only be achieved if one can choose $X_1=1$. Similarly, if $B_2-B_1=X_2-X_1\in (-2,0)$, one must be allowed to choose $X_1=-1$. This necessity of ``looking into the future" prohibits the existence of weak solutions, but not the construction of path-by-path solutions.

\begin{proof}
\textbf{No weak solution.}

Assume that there exists a weak solution $X$ defined on some filtered probability space $(\Omega,\mathcal{F},\mathbb{F},\mathbb{P})$. By Lemma~\ref{lem:allequations}(\ref{en:b}), ${\mathbb{P}(X_1=1 \text{ or } X_1=-1)=1}$. Assume first that $\mathbb{P}(X_1=1)>0$. Let ${A_1\coloneqq \{ X_{1}=1 \}}$ and $A_2\coloneqq \{B_2-B_1 \in (-2,0)\}$. Then we have ${X_2(\omega) \in (-1,1)}$ for $\omega \in A_1\cap A_2$. As $A_1$ is $\mathcal{F}_1$-measurable and $B_2-B_1$ is independent of $\mathcal{F}_1$, $\mathbb{P}(A_1\cap A_2)>0$.
Hence, after a time shift, we would have a weak solution to the SDE
\begin{equation*}
    dX_t= \left(- \mathbbm{1}_{\{X_t\neq 0,|X_t|\leqslant 1\}} \frac{1}{2X_t}+ \mathbbm{1}_{\{X_t> 1\}} \frac{1}{X_t-1} + \mathbbm{1}_{\{X_t < -1\}} \frac{1}{X_t+1}\right) dt + dB_t,
\end{equation*}
on $[0,1]$ with $\mathbb{P}(X_0 \in (-1,1))>0$, which contradicts Lemma \ref{lem:noweaksolutiondirect2}. Hence, $X_1=-1$ a.s. By the same arguments as above we must have $X_1(\omega)=1$ for $\omega \in \tilde{A}_2$, where ${\tilde{A}_2\coloneqq \{B_2-B_1 \in (0,2)}\}$, which gives a contradiction as $\mathbb{P}(\tilde{A}_2)>0$.

\textbf{Existence of a path-by-path solution.}

Let $X^{+}, X^{-}\colon\Omega \rightarrow \mathcal{C}([0,1])$ be the nonnegative and nonpositive Bessel bridge with terminal value $1$, respectively $-1$. Let $C_1\coloneqq \{B_2-B_1>0\}$ and $C_2\coloneqq \{B_2-B_1<0\}$. Define $X\colon\Omega \rightarrow \mathcal{C}([0,2])$ such that, for $t \in [0,2]$,
\[
    X_t(\omega)\coloneqq\begin{dcases}
        X^{+}_{t\wedge 1}(\omega)+B_{t \vee 1}(\omega)-B_1(\omega) & \text{ if }\omega \in C_1, \\
        X^{-}_{t\wedge 1}(\omega)+B_{t \vee 1}(\omega)-B_1(\omega)& \text{ if } \omega \in C_2. \\
    \end{dcases}
\]

Hence $|X_2(\omega)|>1$ for $\omega \in C_1 \cup C_2$. By Lemma~\ref{lem:allequations}(\ref{en:unique}) (after a shift of the space variable), we can uniquely extend $X$ to $[0,3]$ fulfilling \eqref{eq:eqofinterest} so that $|X_t|>1$ for all $t \in [2,3]$.  As $C_1\cup C_2$ has full measure, $X$ is indeed a path-by-path solution. 

\textbf{Existence of other path-by-path solutions.}

Other path-by-path solutions can be constructed. Indeed on the set $C_3\coloneqq \{B_2-B_1>2\}$ we can choose freely if $X$ coincides with the nonnegative or nonpositive Bessel bridge on $[0,1]$ since, for $\omega \in C_3$, $X_2(\omega)>1$. Then, we can proceed the same way as before.
\end{proof}

\subsection{Pathwise unique weak solution, no path-by-path uniqueness}

The following proposition gives an example of a one-dimensional SDE with a pathwise unique weak solution, but path-by-path uniqueness does not hold. Again, the construction can easily be extended to multiple dimensions.

\begin{proposition} \label{lem:counterexample2}
Consider the SDE
\begin{equation}\label{eq:eqofinterest2}
    dX_t= b(t,X_t) dt + dB_t, \quad X_0=0, \quad t \in [0,4],
\end{equation}
where

\[
    b(t,x)=\begin{dcases}
        \mathbbm{1}_{\{x>0\}}\left(\frac{2-x}{1-t}+\frac{1}{x}\right) + \mathbbm{1}_{\{x<0\}}\left(\frac{-2-x}{1-t}+\frac{1}{x}\right) & \text{ if }0\leqslant t < 1, \\
        \mathbbm{1}_{\{x < 0\}}\frac{1}{x}+\mathbbm{1}_{\{x>2\}}\left(\frac{3-x}{2-t}+\frac{1}{x-2}\right) + \mathbbm{1}_{\{0<x<2\}}\left(\frac{1-x}{2-t}+\frac{1}{x-2}\right)& \text{ if }1 \leqslant t < 2, \\
        \mathbbm{1}_{\{x< 0\}}\frac{1}{x} & \text{ if } 2 \leqslant t <3,\\
        \mathbbm{1}_{\{x < 0\}}\frac{1}{x}-\mathbbm{1}_{\{x \neq 2, |x-2|\leqslant 1\}}\frac{1}{2(x-2)}+ \mathbbm{1}_{\{x> 3\}} \frac{1}{x-3} + \mathbbm{1}_{\{0<x <1\}} \frac{1}{x-1} &\text{ if } 3\leqslant t \leqslant 4. \\
    \end{dcases}
\]

There exists a pathwise unique weak solution to \eqref{eq:eqofinterest2}, but path-by-path uniqueness does not hold. 

\end{proposition}

\textbf{Idea of the proof.}
On the time interval $[0,1]$ the drift is the one of an SDE solved by a Bessel 
bridge. Hence, for $t \in [0,1]$ there exists a pathwise unique nonnegative weak 
solution and a pathwise unique nonpositive weak solution. On $[1,4]\times \mathbb{R}_{-}$ 
the drift is constructed such that the nonpositive Bessel bridge on $[0,1]$ can be extended in 
a unique way to a weak solution on the time interval $[0,4]$. On $[1,4]\times 
\mathbb{R}_{+}$ the drift is constructed so that we can extend the nonnegative Bessel 
bridge on the time interval $[0,1]$ to path-by-path solutions on $[0,4]$, but not to a weak solution. Hereby 
solutions are allowed to enter the negative half plane (which is no problem as the drift there ensures ensures the existence of nonpositive solution), but with positive probability a situation as in 
Proposition~\ref{lem:counterexample1} occurs; i.e. weak solutions $X$ are forced to fulfill 
$X_3 \in (1,3)$ with positive probability and on the time interval $[3,4]$ 
the drift is constructed such that these solutions cannot be extended to $[0,4]$.

\begin{proof}

\textbf{Existence of a pathwise unique nonnegative weak solution.}

Note that there exists a pathwise unique nonpositive weak solution $\tilde{X}$ on $[0,1]$ to \eqref{eq:eqofinterest2} with $\tilde{X}_1=-2$ by Lemma~\ref{lem:allequations}\eqref{en:b}. We can extend this solution in a pathwise unique way to $[0,4]$ by Lemma~\ref{lem:allequations}\eqref{en:unique}.

\textbf{No other weak solution.}

Assume that there exists another weak solution $X$ to \eqref{eq:eqofinterest2}. Then we must have $\mathbb{P}(X_1=2)>0$. Let
\begin{align*}
A_1\coloneqq \{X_1=2\}, A_2\coloneqq \{\inf_{t \in [1,2]} X_t>0, X_2=1\}, A_3\coloneqq \{\inf_{t \in [2,3]} (B_t-B_2)>-1, B_3-B_2 \in (0,2)\}.
\end{align*}

First assume that $\mathbb{P}(X_2 = 1 \mid X_1=2)=1$. Then by assumption and by Remark~\ref{rem:upperbound}, for $\varepsilon>0$ small enough,
\begin{align*}
\mathbb{P}(A_1\cap A_2)&\geqslant \mathbb{P}(A_1\cap \{\sup_{t \in [1,2]} |B_t-B_1|<\varepsilon\})\\
&=\mathbb{P}(A_1)\mathbb{P}(\sup_{t \in [1,2]} |B_t-B_1|<\varepsilon)>0.
\end{align*}

Hence $\mathbb{P}(\bigcap_{i=1}^{3} A_i)>0$ as $A_3$ is independent of $A_1\cap A_2$. Note that $X_3(\omega) \in (1,3)$ for ${\omega \in \bigcap_{i=1}^{3} A_i}$. This gives a contradiction to Lemma~\ref{lem:noweaksolutiondirect2} as after a space shift we would have a weak solution to Equation \eqref{eq:nosolutiondirect2} with initial condition $X_0$ fulfilling $\mathbb{P}(X_0 \in (-1,1))>0$. 

Assume now that $\mathbb{P}(X_2=1 \mid X_1=2)<1$ and therefore $\mathbb{P}(X_2=3 \mid X_1=2)>0$. Let $\tilde{A}_1\coloneqq A_1$, $\tilde{A}_2\coloneqq \{X_2=3\}$ and
\begin{equation*}
\tilde{A}_3\coloneqq \{\inf_{t \in [2,3]} (B_t-B_2)>-3, B_3-B_2 \in (-2,0)\}.
\end{equation*}

Then by assumption $\mathbb{P}(\bigcap_{i=1}^{3} \tilde{A}_i)>0$ and $X_3(\omega) \in (1,3)$ for $\omega \in \bigcap_{i=1}^{3} \tilde{A}_i$ and again this leads to a contradiction to Lemma~\ref{lem:noweaksolutiondirect2}.

\textbf{Construction of another path-by-path solution.}

By separately considering the sets $C_1\coloneqq \{B_3-B_2>0\}$ and $C_2 \coloneqq \{B_3-B_2<0\}$ we construct an additional path-by-path solution on the set $C_1\cup C_2$ of full measure. Let $X\colon C_1 \rightarrow \mathcal{C}([0,2])$ coincide with the Bessel bridges such that $X_1=2$ and $X_2=3$. Let $\tau\coloneqq \inf\{t\geqslant 2: B_t-B_2=-3\}$. Then, by Lemma~\ref{lem:allequations}\eqref{en:unique}, for $\omega \in C_1 \cap \{\tau>3\}$ we can extend $X$ to the time interval $[0,4]$ fulfilling equation \eqref{eq:eqofinterest2}.
Now consider $\omega \in C_1 \cap \{\tau\leqslant 3\}$. We can clearly extend $X$ to the time interval $[0,\tau]$ by adding the Brownian increment $B_t-B_2$. Then we have that $X_{\tau}(\omega)=0$. Hence, for $\omega \in C_1 \cap \{\tau\leqslant 3\}$ and $t \in [\tau,4]$, we can choose $X$ to be the pathwise unique nonpositive solution to
\begin{align*}
dX_t= \frac{1}{X_t} dt + dB_t, \ t \in [\tau,4]
\end{align*}
and therefore we can construct $X\colon C_1 \rightarrow \mathcal{C}([0,4])$ fulfilling equation \eqref{eq:eqofinterest2}.

Let
\begin{equation*}
\tilde{b}(t,x)\coloneqq b(t,x)-\mathbbm{1}_{\{t\geqslant 1, x<0\}}\frac{1}{x}+\mathbbm{1}_{\{t\geqslant 3, x \leqslant 0\}} \frac{1}{x-1}.
\end{equation*}
For $\omega \in C_2$ and $t \in [0,4]$, let $\tilde{X}$ fulfill
\begin{equation*}
\tilde{X}_t(\omega)=\int_0^t \tilde{b}(s,\tilde{X}_s(\omega))ds + B_t(\omega)
\end{equation*}
such that $\tilde{X}$ coincides with the two Bessel bridges on $[0,1]$ and $[1,2]$ so that  $\tilde{X}_1=2$ and $\tilde{X}_2=1$. This is possible by the same arguments as in the proof of Proposition~\ref{lem:counterexample1}. Let ${{\tau_0}\coloneqq  \inf\{t >1:\tilde{X}_t=0\}\wedge 4}$. On $[0,{\tau_0}]\cap [0,4]$ let $\hat{X}\coloneqq \tilde{X}$. Then $\hat{X}$ is a solution to \eqref{eq:eqofinterest2} on $[0,{\tau_0}]$ as, for $x> 0$, $b(t,x)=\tilde{b}(t,x)$. On $[{\tau_0},4]$, we choose $\hat{X}$ to be the pathwise unique nonpositive solution to
\begin{equation*}
dX_t=\frac{1}{X_t} dt + dB_t,\ t \in [{\tau_0},4].
\end{equation*}
As $\mathbb{P}(C_1\cup C_2)=1$, we can construct a path-by-path solution $Y\colon \Omega \rightarrow \mathcal{C}([0,4])$ by setting
\[
    Y(\omega)=\begin{dcases}
        X(\omega) & \text{ if }\omega \in C_1, \\
        \hat{X}(\omega)& \text{ if } \omega \in C_2. \\
    \end{dcases}
\]
\end{proof}


\begin{thebibliography}{23}
\providecommand{\natexlab}[1]{#1}
\providecommand{\url}[1]{\texttt{#1}}
\expandafter\ifx\csname urlstyle\endcsname\relax
  \providecommand{\doi}[1]{doi: #1}\else
  \providecommand{\doi}{doi: \begingroup \urlstyle{rm}\Url}\fi

\bibitem[Alabert and Le\'{o}n(2017)]{Alabert}
A.~Alabert and J.~A. Le\'{o}n.
\newblock On uniqueness for some non-{L}ipschitz {SDE}.
\newblock \emph{J. Differential Equations}, 262\penalty0 (12):\penalty0
  6047--6067, 2017.

\bibitem[Anzeletti et~al.(2021)Anzeletti, Richard, and Tanr\'{e}]{AnRiTa}
L.~Anzeletti, A.~Richard, and E.~Tanr\'{e}.
\newblock Regularisation by fractional noise for one-dimensional differential
  equations with nonnegative distributional drift.
\newblock \emph{ArXiv Preprint
  \href{https://arxiv.org/abs/2112.05685}{arXiv:2112.05685}}, 2021.

\bibitem[Bass and Chen(2001)]{BassChen}
R.~F. Bass and Z.-Q. Chen.
\newblock Stochastic differential equations for {D}irichlet processes.
\newblock \emph{Probab. Theory Related Fields}, 121\penalty0 (3):\penalty0
  422--446, 2001.

\bibitem[Beck et~al.(2019)Beck, Flandoli, Gubinelli, and
  Maurelli]{BeckFlandoli}
L.~Beck, F.~Flandoli, M.~Gubinelli, and M.~Maurelli.
\newblock Stochastic {ODE}s and stochastic linear {PDE}s with critical drift:
  regularity, duality and uniqueness.
\newblock \emph{Electron. J. Probab.}, 24:\penalty0 Paper No. 136, 72, 2019.

\bibitem[Butkovsky and Mytnik(2019)]{ButMyt}
O.~Butkovsky and L.~Mytnik.
\newblock Regularization by noise and flows of solutions for a stochastic heat
  equation.
\newblock \emph{Ann. Probab.}, 47\penalty0 (1):\penalty0 165--212, 2019.

\bibitem[Catellier and Gubinelli(2016)]{CatellierGubinelli}
R.~Catellier and M.~Gubinelli.
\newblock Averaging along irregular curves and regularisation of {ODE}s.
\newblock \emph{Stochastic Process. Appl.}, 126\penalty0 (8):\penalty0
  2323--2366, 2016.

\bibitem[Cherny(2000)]{Cherny2000}
A.~S. Cherny.
\newblock On the strong and weak solutions of stochastic differential equations
  governing {B}essel processes.
\newblock \emph{Stochastics}, \penalty0 (3-4):\penalty0 213--219, 2000.

\bibitem[Cherny(2002)]{Cherny2002}
A.~S. Cherny.
\newblock On the uniqueness in law and the pathwise uniqueness for stochastic
  differential equations.
\newblock \emph{Theory Probab. Appl.}, 46, 01 2002.

\bibitem[Davie(2007)]{Davie}
A.~M. Davie.
\newblock Uniqueness of solutions of stochastic differential equations.
\newblock \emph{Int. Math. Res. Not. IMRN}, \penalty0 (24), 2007.

\bibitem[Delarue and Diel(2016)]{DelarueDiel}
F.~Delarue and R.~Diel.
\newblock Rough paths and 1d {SDE} with a time dependent distributional drift:
  application to polymers.
\newblock \emph{Probab. Theory Related Fields}, 165\penalty0 (1-2):\penalty0
  1--63, 2016.

\bibitem[Flandoli(2011)]{Flandoli}
F.~Flandoli.
\newblock \emph{Random perturbation of {PDE}s and fluid dynamic models}, volume
  2015 of \emph{Lecture Notes in Mathematics}.
\newblock Springer, Heidelberg, 2011.
\newblock Lectures from the 40th Probability Summer School held in Saint-Flour,
  2010, \'{E}cole d'\'{E}t\'{e} de Probabilit\'{e}s de Saint-Flour.
  [Saint-Flour Probability Summer School].

\bibitem[Flandoli et~al.(2003)Flandoli, Russo, and Wolf]{FlandoliRussoWolf}
F.~Flandoli, F.~Russo, and J.~Wolf.
\newblock Some {SDE}s with distributional drift. {I}. {G}eneral calculus.
\newblock \emph{Osaka J. Math.}, 40\penalty0 (2):\penalty0 493--542, 2003.

\bibitem[Flandoli et~al.(2017)Flandoli, Issoglio, and Russo]{FIR}
F.~Flandoli, E.~Issoglio, and F.~Russo.
\newblock Multidimensional stochastic differential equations with
  distributional drift.
\newblock \emph{Trans. Amer. Math. Soc.}, 369\penalty0 (3):\penalty0
  1665--1688, 2017.

\bibitem[Galeati and Gubinelli(2022)]{GaleatiGubinelli}
L.~Galeati and M.~Gubinelli.
\newblock Noiseless regularisation by noise.
\newblock \emph{Rev. Mat. Iberoam.}, 38\penalty0 (2):\penalty0 433--502, 2022.

\bibitem[Krylov and R\"{o}ckner(2005)]{KrylovRockner}
N.~V. Krylov and M.~R\"{o}ckner.
\newblock Strong solutions of stochastic equations with singular time dependent
  drift.
\newblock \emph{Probab. Theory Related Fields}, 131\penalty0 (2):\penalty0
  154--196, 2005.

\bibitem[Le~Gall(1984)]{LeGall}
J.-F. Le~Gall.
\newblock One-dimensional stochastic differential equations involving the local
  times of the unknown process.
\newblock In \emph{Stochastic analysis and applications ({S}wansea, 1983)},
  volume 1095 of \emph{Lecture Notes in Math.}, pages 51--82. Springer, Berlin,
  1984.

\bibitem[Pitman(1999)]{Pitman}
J.~Pitman.
\newblock The {SDE} solved by local times of a {B}rownian excursion or bridge
  derived from the height profile of a random tree or forest.
\newblock \emph{Ann. Probab.}, 27\penalty0 (1):\penalty0 261--283, 1999.

\bibitem[Protter(2004)]{Protter}
P.~E. Protter.
\newblock \emph{Stochastic integration and differential equations}, volume~21
  of \emph{Applications of Mathematics (New York)}.
\newblock Springer-Verlag, Berlin, second edition, 2004.
\newblock Stochastic Modelling and Applied Probability.

\bibitem[Revuz and Yor(1999)]{RevuzYor}
D.~Revuz and M.~Yor.
\newblock \emph{Continuous martingales and {B}rownian motion}, volume 293 of
  \emph{Grundlehren der Mathematischen Wissenschaften [Fundamental Principles
  of Mathematical Sciences]}.
\newblock Springer-Verlag, Berlin, third edition, 1999.

\bibitem[Shaposhnikov(2016)]{Shapo}
A.~V. Shaposhnikov.
\newblock Some remarks on {D}avie's uniqueness theorem.
\newblock \emph{Proc. Edinb. Math. Soc. (2)}, 59\penalty0 (4):\penalty0
  1019--1035, 2016.

\bibitem[Shaposhnikov(2017)]{shapocorrection}
A.~V. Shaposhnikov.
\newblock Correction to the paper "some remarks on davie's uniqueness theorem".
\newblock \emph{ArXiv Preprint
  \href{https://arxiv.org/abs/1703.06598}{arXiv:1703.06598}}, 2017.

\bibitem[Shaposhnikov and Wresch(2020)]{ShaWre}
A.~V. Shaposhnikov and L.~Wresch.
\newblock Pathwise vs. path-by-path uniqueness.
\newblock \emph{ArXiv Preprint
  \href{https://arxiv.org/abs/2001.02869}{arXiv:2001.02869}}, 2020.

\bibitem[Veretennikov(1980)]{Veretennikov}
A.~J. Veretennikov.
\newblock Strong solutions and explicit formulas for solutions of stochastic
  integral equations.
\newblock \emph{Mat. Sb. (N.S.)}, 111(153)\penalty0 (3):\penalty0 434--452,
  480, 1980.

\end{thebibliography}
\end{document}